\documentclass[12pt,reqno]{amsart}
\usepackage{amsmath, amsthm, amssymb, stmaryrd}
\usepackage{hyperref}
\usepackage{enumerate}
\usepackage{url}
\usepackage{xcolor}  	

\let\OLDthebibliography\thebibliography
\renewcommand\thebibliography[1]{
  \OLDthebibliography{#1}
  \setlength{\parskip}{0pt}
  \setlength{\itemsep}{0pt plus 0.3ex}
}

\usepackage{mathtools}

\usepackage{multirow, array}
\usepackage{placeins}

\usepackage{caption} 
\advance \topmargin by -\headheight
\advance \topmargin by -\headsep
     
\evensidemargin \oddsidemargin
\newtheorem{thm}{Theorem}[section]
\newtheorem{lemma}[thm]{Lemma}

\newtheorem{cor}[thm]{Corollary}
\newtheorem{conj}[thm]{Conjecture}

\theoremstyle{definition}

\theoremstyle{remark}

\numberwithin{equation}{section}

\newcommand{\mmod}[1]{{\,\,\mathrm{mod}\,\,#1}}

\newcommand*\wrapletters[1]{\wr@pletters#1\@nil}
\def\wr@pletters#1#2\@nil{#1\allowbreak\if&#2&\else\wr@pletters#2\@nil\fi}

\def\alp{{\alpha}} 
\def\bet{{\beta}}  
\def\gam{{\gamma}} 
\def\del{{\delta}}

\def\lam{{\lambda}}

\def\ome{{\omega}}  
\def\eps{\varepsilon}

\def\le{\leqslant} \def\ge{\geqslant}  
\def \leq {\leqslant} \def \geq {\geqslant}

\def \bN {\mathbb N}

\def \bQ {\mathbb Q}
\def \bR {\mathbb R}
\def \bZ {\mathbb Z}

\def \bn {\mathbf n}

\def \bv {\mathbf v}
\def \bx {\mathbf x}

\def \by {\mathbf y}

\def \bzero {\mathbf 0}

\def \balp {{\boldsymbol{\alp}}}

\def \bgam {{\boldsymbol{\gam}}}
\def \bdel {{\boldsymbol{\del}}}

\def \cB {\mathcal B}

\def \cH {\mathcal H}

\def \cJ {\mathcal J}

\def \cN {\mathcal N}

\def \cR {\mathcal R}
\def \cS {\mathcal S}

\def \dim {\mathrm{dim}}

\def \det {\mathrm{det}}

\def \vol {\mathrm{vol}}

\def \sgn {{\mathrm{sgn}}}

\begin{document}
\title[Higher-rank Bohr sets and diophantine approximation]{Higher-rank Bohr sets and multiplicative diophantine approximation}
\author[Sam Chow]{Sam Chow}
\address{Mathematical Institute, University of Oxford, Andrew Wiles Building, Radcliffe Observatory Quarter,
Woodstock Road, Oxford OX2 6GG, United Kingdom; and Department of Mathematics, University of York, Heslington, York, YO10 5DD, United Kingdom}
\email{sam.chow@maths.ox.ac.uk}
\author[Niclas Technau]{Niclas Technau}
\address{Department of Mathematics, University of York, Heslington, York, YO10 5DD, United Kingdom}
\email{niclas.technau@york.ac.uk}
\subjclass[2010]{11J83, 11J20, 11H06, 52C07}
\keywords{Metric diophantine approximation, geometry of numbers, additive combinatorics}
\thanks{}
\date{}
\begin{abstract} Gallagher's theorem is a sharpening and extension of the Littlewood conjecture that holds for almost all tuples of real numbers. We provide a fibre refinement, solving a problem posed by Beresnevich, Haynes and Velani in 2015. Hitherto, this was only known on the plane, as previous approaches relied heavily on the theory of continued fractions. Using reduced successive minima in lieu of continued fractions, we develop the structural theory of Bohr sets of arbitrary rank, in the context of diophantine approximation. In addition, we generalise the theory and result to the inhomogeneous setting. To deal with this inhomogeneity, we employ diophantine transference inequalities in lieu of the three distance theorem.
\end{abstract}
\maketitle

\section{Introduction}
\label{intro}

\subsection{Results}

The Littlewood conjecture (circa 1930) is perhaps the most sought-after result in diophantine approximation. It asserts that if $\alp, \bet \in \bR$ then 
\[
\liminf_{n \to \infty} n \| n \alp \| \cdot \| n \bet \| = 0.
\]
However, Gallagher's theorem \cite{Gal1962} implies that if $k \ge 2$ then for almost all tuples $(\alp_1, \ldots, \alp_k) \in \bR^k$ the stronger statement
\begin{equation} \label{GallagherConclusion}
\liminf_{n \to \infty} n (\log n)^k \| n \alp_1 \| \cdots \| n \alp_k \| = 0
\end{equation}
is valid. Beresnevich, Haynes and Velani \cite[Theorem 2.1 and Remark 2.4]{BHV2016} showed that if $k = 2$ then for any $\alp_1 \in \bR$ the statement \eqref{GallagherConclusion} holds for almost all $\alp_2 \in \bR$. On higher-dimensional fibres, the problem has remained visibly open until now \cite[Problem 2.1]{BHV2016}. We solve this problem.

\begin{thm} \label{solution}
If $k \ge 2$ and $\alp_1, \ldots, \alp_{k-1} \in \bR$ then for almost all $\alp_k \in \bR$ we have
\[
\liminf_{n \to \infty} n (\log n)^k \| n \alp_1 \| \cdots \|n \alp_k \| = 0.
\]
\end{thm}

What we show is more general. The \emph{multiplicative exponent} of the vector $\balp = (\alp_1, \ldots, \alp_{k-1})$, denoted $\omega^\times(\balp)$, is the supremum of the set of real numbers $w$ such that, for infinitely many $n \in \bN$, we have
\[
\| n \alp_1 \| \cdots \| n \alp_{k-1} \| < n^{-w}.
\]

\begin{thm} \label{main} Let $k \ge 2$, let $\alp_1, \ldots, \alp_{k-1}, \gam_1, \ldots, \gam_{k-1} \in \bR$, and assume that the multiplicative exponent of $\balp = (\alp_1, \ldots, \alp_{k-1})$ satisfies $\omega^\times(\balp) < \frac{k-1}{k-2}$. Let $\psi: \bN \to \bR_{\ge 0}$ be a decreasing function such that 
\begin{equation} \label{divergence}
\sum_{n=1}^\infty \psi(n) (\log n)^{k-1} = \infty.
\end{equation}
Then for almost all $\alp \in \bR$ there exist infinitely many $n \in \bN$ such that
\[
\| n \alp_1 - \gam_1 \| \cdots \| n \alp_{k-1} - \gam_{k-1} \| \cdot  \| n \alp \| < \psi(n).
\]
\end{thm}

\noindent The $k=2$ case was established in \cite{Cho2018}. In that case, the condition becomes $\omega^\times(\alp) < \infty$, which is equivalent to $\alp$ being irrational and non-Liouville.

Theorem \ref{solution} will follow from Theorem \ref{main}, except in the case $\omega^\times(\balp) \ge \frac{k-1}{k-2}$. However, in the latter case there exist arbitrarily large $n \in \bN$ for which
\[
\| n\alp_1 \| \cdots \| n \alp_{k-1} \| < n^{-1 - \frac1{k-1}}.
\]
For these $n$, we thus have
\[
n (\log n)^k \| n \alp_1 \| \cdots \| n \alp_k \| = o(1)
\]
for all $\alp_k$. This completes the deduction of Theorem \ref{solution} assuming Theorem \ref{main}.

The hypothesis $\omega^\times(\balp) < \frac{k-1}{k-2}$ is generic. Indeed, it follows directly from the work of Hussain and Simmons \cite[Corollary 1.4]{HS2018}, or alternatively from the prior but weaker conclusions of \cite[Remark 1.2]{BV2015}, that the exceptional set 
\[
\Bigl \{ \balp \in \bR^{k-1}: \omega^\times(\balp) \ge \frac{k-1}{k-2} \Bigr \}
\]
has Hausdorff codimension $\frac1{2k-3}$ in $\bR^{k-1}$. This is much stronger than the assertion that the set of exceptions has Lebesgue measure zero. 

Some condition on $(\alp_1,\ldots,\alp_{k-1})$ is needed for Theorem \ref{main}. For example, if $(\alp_1, \ldots, \alp_{k-1}) \in \bQ^{k-1}$, then the $n \alp_i$ take on finitely many values modulo 1, so if the $\gam_i$ avoid these then 
\[
\| n \alp_1 - \gam_1 \| \cdots \| n \alp_{k-1} - \gam_{k-1} \| \gg 1.
\]
Khintchine's theorem (Theorem \ref{Khintchine}) then refutes the conclusion of Theorem \ref{main} in this scenario, for appropriate $\psi$.

Theorem \ref{main} is sharp, in the sense that the divergence hypothesis \eqref{divergence} is necessary, as we now explain. Gallagher's work \cite{Gal1962} shows, more precisely, the following (see \cite[Remark 1.2]{BV2015}).

\begin{thm} [Gallagher] \label{Gallagher}
Let $k \ge 2$, and write $\mu_k$ for $k$-dimensional Lebesgue measure. Let $\psi: \bN \to \bR_{\ge 0}$ be a decreasing function, and denote by $W_k^\times(\psi)$ the set of $(\alp_1, \ldots, \alp_k) \in [0,1]^k$ such that
\[
\| n \alp_1 \| \cdots \| n \alp_k \| < \psi(n)
\]
has infinitely many solutions $n \in \bN$. Then
\[
\mu_k(W_k^\times(\psi)) = \begin{cases}
0, &\text{if } \sum_{n=1}^\infty \psi(n) (\log n)^{k-1} < \infty \\
1, &\text{if } \sum_{n=1}^\infty \psi(n) (\log n)^{k-1} = \infty.
\end{cases}
\]
\end{thm}

\noindent In particular, the divergence part of this statement is sharp. Theorem \ref{main} is stronger still, so it must also be sharp, insofar as it is necessary to assume \eqref{divergence}.

Some readers may be curious about the multiplicative Hausdorff theory. Owing to the investigations of Beresnevich--Velani \cite[\S1]{BV2015} and Hussain--Simmons \cite{HS2018}, we now understand that genuine `fractal' Hausdorff measures are insensitive to the multiplicative nature of such problems. With $k \in \bZ_{\ge 2}$ and $\psi: \bN \to \bR_{\ge 0}$, let $W_k^\times(\psi)$ be as in Theorem \ref{Gallagher}, and denote by $W_k(\psi)$ the set of $(\alp_1,\ldots,\alp_k) \in [0,1]^k$ for which 
\[
\max( \| n \alp_1 \|, \ldots, \| n \alp_k \| ) < \psi(n)
\]
has infinitely many solutions $n \in \bN$. In light of \cite[Corollary 1.4]{HS2018} and \cite[Theorem 4.12]{BRV2016}, we have the Hausdorff measure identity
\[
\cH^s( W_k^\times (\psi)) = \cH^{s- (k-1)} (W_1(\psi)) \qquad (k-1 < s < k).
\]
Loosely speaking, this reveals that multiplicatively approximating $k$ real numbers at once is the same as approximating one of the $k$ numbers, save for a set of zero Hausdorff $s$-measure. This is in stark contrast to the Lebesgue case $s=k$, wherein there are extra logarithms in the multiplicative setting (compare Theorems \ref{Gallagher} and \ref{Khintchine}). As discussed in \cite{BV2015, HS2018}, if $s > k$ then $\cH^s( W_k^\times (\psi)) = 0$, irrespective of $\psi$, while if $s \le k-1$ then $\cH^s( W_k^\times (\psi)) = \infty$, so long as $\psi$ does not vanish identically.

\subsection{Techniques}
\label{techniques}

The proof of Theorem \ref{main} parallels \cite{Cho2018}, with a more robust approach needed for the structural theory of Bohr sets. Recalling that
\[
\balp = (\alp_1, \ldots, \alp_{k-1}), \qquad \bgam = (\gam_1, \ldots, \gam_{k-1})
\]
are fixed, we introduce the auxiliary approximating function $\Phi = \Phi_\balp^\bgam$ given by
\begin{equation} \label{PhiDef}
\Phi(n) = \frac{\psi(n)}{\|n \alp_1 - \gam_1 \| \cdots \| n \alp_{k-1} - \gam_{k-1} \| }.
\end{equation}
The conclusion of Theorem \ref{main} is equivalent to the assertion that for almost all $\alp \in \bR$ there exist infinitely many $n \in \bN$ such that
\[
\| n \alp \| < \Phi(n).
\]
If $\Phi$ were monotonic, then Khintchine's theorem \cite[Theorem 2.3]{BRV2016} would be a natural and effective approach.

\begin{thm}[Khintchine's theorem] \label{Khintchine}
Let $\Phi: \bN \to \bR_{\ge 0}$. Then the measure of the set
\[
\{ \alp \in [0,1]: \| n \alp \| < \Phi(n) \text{ for infinitely many } n \in \bN \}
\]
is 
\[
\begin{cases} 
0, &\text{if } \sum_{n=1}^\infty \Phi(n) < \infty \\
1, &\text{if } \sum_{n=1}^\infty \Phi(n) = \infty \text{ and } \Phi \text{ is monotonic}.
\end{cases}
\]
\end{thm}

For any $n \in \bN$ the function $\alp \mapsto \| n \alp \|$ is periodic modulo 1, so Khintchine's theorem implies that if $\Phi$ is monotonic and $\sum_{n=1}^\infty \Phi(n) = \infty$ then for almost all $\bet \in \bR$ the inequality $\| n \bet \| < \Phi(n)$ holds for infinitely many $n \in \bN$. The specific function $\Phi$ defined in \eqref{PhiDef} is very much \textbf{not} monotonic, so for Theorem \ref{main} our task is more demanding. We place the problem in the context of the Duffin--Schaeffer conjecture \cite{DS1941}.

\begin{conj}[Duffin--Schaeffer conjecture, 1941]\label{DSconj}
Let $\Phi: \bN \to \bR_{\ge 0}$ satisfy 
\begin{equation} \label{DShyp}
\sum_{n=1}^\infty \frac{\varphi(n)}n \Phi(n) = \infty.
\end{equation}
Then for almost all $\bet \in \bR$ the inequality
\[
| n \bet - r | < \Phi(n)
\]
holds for infinitely many coprime pairs $(n,r) \in \bN \times \bZ$.
\end{conj}

\noindent For comparison to Khintchine's theorem, note that if $\Phi$ is monotonic then the divergence of $\sum_{n=1}^\infty \frac{\varphi(n)}n \Phi(n)$ is equivalent to that of $\sum_{n=1}^\infty \Phi(n)$.

The Duffin--Schaeffer conjecture has stimulated research in diophantine approximation for decades, and remains open. There has been some progress, including the Erd\H{o}s--Vaaler theorem \cite[Theorem 2.6]{Har1998}, as well as \cite{Aist2014, BHHV2013, HPV2012} and, most recently, \cite{Fufu}. For our purpose, the most relevant partial result is the Duffin--Schaeffer theorem \cite[Theorem 2.5]{Har1998}.

\begin{thm}[Duffin--Schaeffer theorem] \label{DSthm}
Conjecture \ref{DSconj} holds under the additional hypothesis 
\begin{equation} \label{additional}
\limsup_{N \to \infty} \Bigl( \sum_{n \le N} \frac{\varphi(n)}n \Phi(n) \Bigr) \Bigl( \sum_{n \le N} \Phi(n) \Bigr)^{-1} > 0.
\end{equation}
\end{thm}

\noindent Here $\varphi$ is the Euler totient function, given by $\varphi(n) = \displaystyle \sum_{\substack{a \le n \\ (a,n) = 1}} 1$. 

If $\Phi$ were supported on primes, for instance, then the hypothesis \eqref{additional} would present no difficulties \cite[p. 27]{Har1998}, but in general this hypothesis is quite unwieldy. There have been very few genuinely different examples in which the Duffin--Schaeffer theorem has been applied but, as demonstrated in \cite{Cho2018}, approximating functions of the shape $\Phi_\balp^\gam$ are susceptible to this style of attack.

We tame our auxiliary approximating function $\Phi$ by restricting its support to a `well-behaved' set $G$, giving rise to a modified auxiliary approximating function $\Psi = \Psi_\balp^\bgam$ (see \S\S \ref{fractional} and \ref{final}). The Duffin--Schaeffer theorem will be applied to $\Psi$. By partial summation and the monotonicity of $\Psi$, we are led to estimate the sums
\begin{equation} \label{T}
T_N(\balp, \bgam) := \sum_{\substack{n \le N \\ n \in G}} \frac1{\| n \alp_1 - \gam_1\| \cdots \| n \alp_{k-1} - \gam_{k-1} \|}
\end{equation}
and
\begin{equation} \label{Tstar}
T_N^*(\balp, \bgam) := \sum_{\substack{n \le N \\ n \in G}} \frac{\varphi(n)}{n \| n \alp_1 - \gam_1 \| \cdots \| n \alp_{k-1} - \gam_{k-1} \|}.
\end{equation}

Specifically, we require sharp upper bounds for the first sum and sharp lower bounds for the second. By dyadic pigeonholing, the former boils down to estimating the cardinality of \emph{Bohr sets}
\begin{equation} \label{BohrDef}
B = B_\balp^\bgam (N; \bdel) := \{ n \in \bZ: |n| \le N, \| n \alp_i - \gam_i \| \le \del_i \quad (1 \le i \le k-1) \}.
\end{equation}
The latter, meanwhile, demands that we also understand the structure of $B$; we will be allowed to impose a size restriction on the $\del_i$ to make this work. 

Bohr sets have been studied in other parts of mathematics, notably in additive combinatorics \cite[\S 4.4]{TV2006}. The idea is that there should be generalised arithmetic progressions $P$ and $P'$, of comparable size, for which $P \subseteq B \subseteq P'$. This correspondence is well-understood in the context of abelian groups, but for diophantine approximation the foundations are still being laid. In \cite{Cho2018}, the first author constructed $P$ in the case $k=2$ case using continued fractions, drawing inspiration from Tao's blog post \cite{TaoPost}. Lacking such a theory in higher dimensions, we will use reduced successive minima in this article, and the theory of exponents of diophantine approximation will be used to handle the inhomogeneity. We shall also construct the homogeneous counterpart of $P'$, in order to estimate the cardinality of $B$. 

The basic idea is to lift $B$ to a set $\tilde B \subset \bZ^k$. To determine the structure of $\tilde B$, we procure a discrete analogue of John's theorem, akin to that of Tao and Vu \cite[Theorem 1.6]{TV2008}. The structural data provided in \cite{TV2008} are insufficient for our purposes, as they only assert the upper bound $\dim(P) \le k$. By exercising some control over the parameters, which we may for the problem at hand, we show not only that $\dim(P) = k$, but also that each dimension has substantial length. In addition, we extend to the inhomogeneous case.

As in \cite{Cho2018}, the totient function does average well: we show that $\frac{\varphi(n)}n \gg 1$ on average over our generalised arithmetic progressions. This will eventually enable us to conclude that 
\[
T_N^*(\balp, \bgam) \asymp T_N(\balp, \bgam),
\]
and to then complete the proof of Theorem \ref{main} using the Duffin--Schaeffer theorem.

\subsection{Open problems}

\subsubsection{The large multiplicative exponent case}

It is plausible that Theorem \ref{main} might hold without the assumption $\omega^\times(\balp) < \frac{k-1}{k-2}$; as discussed in the introduction, some assumption is necessary (irrationality, for example). This aspect has not been solved even in the case $k=2$, see \cite{Cho2018}. When $k=2$, the hypothesis $\omega^\times(\balp) < \frac{k-1}{k-2}$ is equivalent to $\alp_1$ being irrational and non-Liouville and, whilst the former is necessary, the latter is likely not.

\subsubsection{The convergence theory}

It would be desirable to have a closer convergence counterpart to Theorem \ref{main}, in the spirit of \cite[Corollary 2.1]{BHV2016}. A homogeneous convergence statement would follow from an upper bound of the shape
\[
\sum_{n \le N} \frac1{ \| n \alp_1 \| \cdots \| n \alp_{k-1}\| } \ll N(\log N)^{k-1}
\]
for the sums considered in \cite{Bug2009, Fre2018, LV2015}, together with an application of the Borel--Cantelli lemma. This bound is generically false \cite[\S 1.2.4]{BHV2016} in the case $k=2$, and when $k \ge 3$ is considered to be difficult to obtain even for a single vector $(\alp_1, \ldots, \alp_{k-1})$; see the question surrounding \cite[Equation (1.4)]{LV2015}. It is likely that the logarithmically-averaged sums
\[
\sum_{n \le N} \frac1{ n \| n \alp_1 \| \cdots \| n \alp_{k-1} \|}
\]
are better-behaved. Perhaps the order of magnitude is generically $(\log N)^k$, as is known when $k=2$ (see \cite[\S 1.2.4]{BHV2016}).

\subsubsection{A special case of the Duffin--Schaeffer conjecture}

In the course of our proof of Theorem \ref{main}, we establish the Duffin--Schaeffer conjecture for a class of functions, namely those modified auxiliary approximating functions of the shape $\Psi = \Psi_\balp^\bgam$. The task of proving the Duffin--Schaeffer conjecture for the unmodified functions $\Phi = \Phi_\balp^\bgam$, however, remains largely open, even in the simplest case $k=2$.

\subsubsection{Inhomogeneous Duffin--Schaeffer problems} 

Inhomogeneous variants of the Duffin--Schaeffer conjecture have received some attention in recent years \cite{BHV2016, Cho2018, Ram2016, Ram2017, Yu2018}. If we knew an inhomogeneous version of the Duffin--Schaeffer theorem, then the following assertion would follow from our method.

\begin{conj} \label{cond} Let $k \ge 2$, let $\alp_1, \ldots, \alp_{k-1}, \gam_1, \ldots, \gam_k \in \bR$, and assume that the multiplicative exponent of $\balp = (\alp_1, \ldots, \alp_{k-1})$ satisfies $\ome^\times(\balp) < \frac{k-1}{k-2}$. Let $\psi: \bN \to \bR_{\ge 0}$ be a decreasing function satisfying \eqref{divergence}. Then for almost all $\alp \in \bR$ there exist infinitely many $n \in \bN$ such that
\[
\| n \alp_1 - \gam_1 \| \cdots \| n \alp_{k-1} - \gam_{k-1} \| \cdot  \| n \alp - \gam_k\| < \psi(n).
\]
\end{conj}

It would follow, for instance, if we knew the following \cite[Conjecture 1.7]{Cho2018}.

\begin{conj} [Inhomogeneous Duffin--Schaeffer theorem] 
Let $\del \in \bR$, and let $\Phi: \bN \to \bR_{\ge 0}$ satisfy \eqref{DShyp} and \eqref{additional}. Then for almost all $\bet \in \bR$ there exist infinitely many $n \in \bN$ such that
\[
\| n \bet - \del \| < \Phi(n).
\]
\end{conj}

\noindent There is little consensus over what the `right' statement of the inhomogeneous Duffin--Schaeffer theorem should be. The assumption \eqref{additional} may not ultimately be necessary, just as it is conjecturally not needed when $\del = 0$. In the inhomogeneous setting, we do not at present even have an analogue of Gallagher's zero-full law \cite{Gal1961}.

\subsubsection{The dual problem} We hope to address this in future work. 

\begin{conj} Let $\alp_1, \ldots, \alp_{k-1} \in \bR$. For $n \in \bZ$ write $n^+ = \max(|n|, 2)$, and define
\begin{align*}
\psi: \bZ_{\ge 2} &\to \bR_{\ge 0} \\
n &\mapsto n^{-1} (\log n)^{-k}.
\end{align*}
Then for almost all $\alp_k \in \bR$ there exist infinitely many $(n_1,\ldots,n_k) \in \bZ^k$ such that
\begin{equation} \label{dual}
 \| n_1 \alp_1 + \cdots + n_k \alp_k \| < \psi(n_1^+ \cdots n_k^+).
\end{equation}
\end{conj}

To motivate this, observe that the conditions $\alp_k \in [0,1]$ and \eqref{dual} define a limit superior set of unions of balls
\[
E_\bn = \bigcup_{a=0}^{n_k} B \Bigl( \frac{a - n_1 \alp_1 - \cdots - n_{k-1} \alp_{k-1}} {n_k}, \frac{ \psi( n_1 \cdots n_k) }{n_k} \Bigr) \cap [0,1].
\]
(Let us assume, for illustration, that $n_1, \ldots, n_k > 0$. This is a simplification of reality.) Using partial summation and the fact that 
\[
\sum_{n \le N} \sum_{n_1 \cdots n_k = n} 1 \asymp_k N (\log N)^{k-1},
\]
one can show that
\[
\sum_{\substack{n_1,\ldots,n_k \in \bN}} \mu(E_\bn) \gg \sum_{n=2}^\infty \frac1{n \log n} = \infty.
\]
In view of the Borel--Cantelli lemmas, we would expect on probabilistic grounds that $\displaystyle \limsup_{\bn \to \infty} E_\bn$ has full measure in $[0,1]$, and one can use periodicity to extend this reasoning to $\alp_k \in \bR$.

\subsection{Organisation}

In \S \ref{exponents}, we recall the relevant diophantine transference inequalities, in particular Khintchine transference and that of Bugeaud--Laurent. Then, in \S \ref{Bohr}, we develop the structural theory of Bohr sets, in this higher-dimensional diophantine approximation setting. This enables us to prove that the Euler totient function averages well on our Bohr sets, in \S \ref{average}, paving the way for us to show that the sums $T_N(\balp, \bgam)$ and $T_N^*(\balp, \bgam)$ are comparable, in \S \ref{fractional}. With all of the ingredients in place, we finish the proof of our main result, Theorem \ref{main}, in \S \ref{final}.

\subsection{Notation}

We use the Bachmann--Landau and Vinogradov notations: for functions $f$ and positive-valued functions $g$, we write $f \ll g$ or $f = O(g)$ if there exists a constant $C$ such that $|f(x)| \le C g(x)$ for all $x$. The constants implied by these notations are permitted to depend on $\alp_1, \ldots, \alp_{k-1}$. Further, we write $f \asymp g$ if $f \ll g \ll f$. If $S$ is a set, we denote the cardinality of $S$ by $|S|$ or $\# S$. The symbol $p$ is reserved for primes. The pronumeral $N$ denotes a positive integer, sufficiently large in terms of $\alp_1, \ldots, \alp_{k-1}$. When $x \in \bR$, we write $\| x \|$ for the distance from $x$ to the nearest integer.

\subsection{Funding and acknowledgments}

The authors were supported by EPSRC Programme Grant EP/J018260/1. SC was also supported by EPSRC Fellowship Grant EP/S00226X/1. We thank Victor Beresnevich, Lifan Guan, Mumtaz Hussain, Antoine Marnat and Terence Tao for beneficial conversations, and Antoine Marnat for introducing us to a broad spectrum of exponents of diophantine approximation. Most of all, SC thanks Victor Beresnevich and Sanju Velani for introducing him to the wonderful world of metric diophantine approximation.

\section{Diophantine exponents and transference inequalities} 
\label{exponents}

Beginning with \emph{Khintchine transference} \cite{BRV2016, Khi1926}, the relationship between simultaneous and dual approximation remains an active topic of research. Our focus will be on the inhomogeneous theory of Bugeaud and Laurent \cite{BL2005}, which builds upon foundational work of Mahler on dual lattices from the late 1930s (see \cite[Corollary 2.3]{Eve2018} and the surrounding discussion). For real vectors $\balp = (\alp_1, \ldots, \alp_d)$ and $\bgam = (\gam_1, \ldots, \gam_d)$, this provides a lower bound for the \emph{uniform simultaneous inhomogeneous exponent} $\hat \omega(\balp, \bgam)$ in terms of the \emph{dual exponent} $\omega^*(\balp)$. There have since been refinements and generalisations by a number of authors, among them Beresnevich--Velani \cite{BV2010}, Ghosh--Marnat \cite{GM2018}, and Chow--Ghosh--Guan--Marnat--Simmons \cite{CGGMS}.

We commence by introducing the \emph{simultaneous exponent} $\omega(\balp)$ of a vector $\balp = (\alp_1,\ldots, \alp_d) \in \bR^d$. This is the supremum of the set of real numbers $w$ such that, for infinitely many $n \in \bN$, we have
\[
\| n \alp_i \| < n^{-w} \qquad (1 \le i \le d).
\]
Comparing this to the multiplicative exponent $\ome^\times(\balp)$ defined in the introduction, it follows immediately from the definitions that
\[
d \ome(\balp) \le \ome^\times(\balp).
\]

For $\balp \in \bR^d$, define $\omega^*(\balp)$ as the supremum of the set of real numbers $w$ such that, for infinitely many $\bn = (n_1, \ldots, n_d) \in \bZ^d$, we have
\[
\| n_1 \alp_1 + \cdots + n_d \alp_d \| \le |\bn|^{-w}.
\]
For $\balp, \bgam \in \bR^d$, define $\hat \omega(\balp, \bgam)$ as the supremum of the set of real numbers $w$ such that, for any sufficiently large real number $X$, there exists $n \in \bN$ satisfying
\[
n < X, \qquad \| n \alp_i - \gam_i \| < X^{-w} \quad (1 \le i \le d).
\]
Below we quote a special case of the main theorem of \cite{BL2005}.

\begin{thm} [Bugeaud--Laurent] \label{BL} If $\balp, \bgam \in \bR^d$ then
\[
\hat \omega(\balp, \bgam) \ge \omega^*(\balp)^{-1}.
\]
\end{thm}

In the context of Theorem \ref{main}, we have $d = k -1$ and 
\[
\omega(\balp) \le \frac{\ome^\times(\balp)}d < \frac1{d-1}.
\]
Khintchine transference \cite[Theorem K]{BL2010} gives
\[
\frac{\omega^*(\balp) } {d + (d-1)\omega^*(\balp)} \le \omega(\balp)  < \frac1{d-1},
\]
and in particular
\begin{equation} \label{NotDualLiouville}
\omega^*(\balp) < \infty.
\end{equation}
Theorem \ref{BL} then furnishes a positive lower bound for $\hat \ome(\balp, \bgam)$, uniform in $\bgam$. In the sequel, let $\eps$ be a positive real number, sufficiently small in terms of $\alp_1, \ldots, \alp_{k-1}$.

\section[Bohr sets]{The structural theory of Bohr sets}
\label{Bohr}

In this section, we develop the correspondence between Bohr sets and generalised arithmetic progressions. In a different context, this is a fundamental paradigm of additive combinatorics \cite{TV2006}. For diophantine approximation, the first author used continued fractions to describe the theory in the case of rank one Bohr sets in \cite{Cho2018}. In the absence of a satisfactory higher-dimensional theory of continued fractions, we take a more general approach here, involving reduced successive minima. Our theory is inhomogeneous, which presents an additional difficulty. To handle this aspect, we deploy the theory of diophantine exponents, specifically Theorem \ref{BL} of Bugeaud--Laurent \cite{BL2005}.

Let $N$ be a large positive integer, and recall that we have fixed $\balp \in \bR^{k-1}$ with $\ome^\times(\balp) < \frac{k-1}{k-2}$. The shift vector $\bgam = (\gam_1, \ldots, \gam_{k-1})$ is also fixed, and for certain values of $\bdel = (\del_1, \ldots, \del_{k-1})$ we wish to study the structure of the Bohr set $B = B_\balp^\bgam (N; \bdel)$ defined in \eqref{BohrDef}. This rank $k-1$ Bohr set $B$ has the structure of a $k$-dimensional generalised arithmetic progression: we construct such patterns $P$ and $P'$ with a number of desirable properties, including that $P \subseteq B \subseteq P'$. For concreteness, we introduce the notations
\[
P(b; A_1, \ldots, A_k; N_1, \ldots, N_k) = \{ b + A_1 n_1 + \cdots + A_k n_k: |n_i| \le N_i \}
\]
and
\[
P^+(b; A_1, \ldots, A_k; N_1, \ldots, N_k) = \{ b + A_1 n_1 + \cdots + A_k n_k: 1 \le n_i \le N_i \},
\]
when $b, A_1, \ldots, A_k, N_1, \ldots, N_k \in \bN$. The latter generalised arithmetic progression is \emph{proper} if for each $n \in P^+(b, A_1, \ldots, A_k, N_1, \ldots, N_k)$ there is a \textbf{unique} vector $(n_1, \ldots, n_k) \in \bN^k$ for which
\[
n_i \le N_i \quad (1 \le i \le k), \qquad n = b + A_1 n_1 + \cdots + A_k n_k.
\]

Most of our structural analysis is based on the geometry of numbers in $\bR^k$. With
\[
\pi_1 : \bR^k \to \bR
\]
being projection onto the first coordinate, observe that $B = \pi_1(\tilde B)$, where
\[
\tilde B = \{ (n, a_1, \ldots, a_{k-1})  \in \bZ^k: |n| \le N, | n \alp_i - \gam_i - a_i| \le \del_i \quad (1 \le i \le k-1) \}.
\]
Meanwhile, our generalised arithmetic progressions will essentially be projections of suitably-truncated lattices. For $\bv_1, \ldots, \bv_k \in \bZ^k$ and $N_1,\ldots,N_k \in \bN$, define
\[
\tilde P (\bv_1, \ldots, \bv_k; N_1, \ldots, N_k) = \{ n_1 \bv_1 + \cdots + n_k \bv_k: |n_i| \le N_i \}.
\]
To orient the reader, we declare in advance that we will choose $A_i = |\pi_1(\bv_i)|$ for all $i$.

Our primary objective in this section is to prove the following lemma.

\begin{lemma} [Inner structure] \label{inner}  Assume
\[
N^{-\eps} \le \del_i \le 1 \quad (1 \le i \le k-1).
\]
Then there exists a proper generalised arithmetic progression 
\[
P = P^+(b; A_1, \ldots, A_k; N_1, \ldots, N_k) 
\]
contained in $B$, for which 
\[
|P| \gg \del_1 \cdots \del_{k-1} N, \qquad N_i \ge N^\eps \quad(1 \le i \le k), \qquad N^{\sqrt \eps} \le b \le \frac N{10}
\]
and
\begin{equation} \label{coprime}
\gcd(A_1,\ldots,A_k) = 1.
\end{equation}
\end{lemma}

\bigskip
Our approach to analysing $\tilde B$ is similar to that of Tao and Vu \cite{TV2008}. Under our hypotheses, we are able to obtain the important inequalities $N_i \ge N^\eps$ ($1 \le i \le k$), and also to deal with the inhomogeneous shift. These two features are not present in \cite{TV2008}, which is more general.

\subsection{Homogeneous structure}

We begin with the homogeneous lifted Bohr set
\[
\tilde B_0 := \Bigl\{ (n, a_1, \ldots, a_{k-1}) \in \bZ^k: |n| \le \frac{N}{10}, | n \alp_i - a_i| \le \frac1{10}\del_i \quad (1 \le i \le k-1) \Bigr\}.
\]
This consists of the lattice points in the region
\[
\cR := \Bigl\{ (n, a_1, \ldots, a_{k-1}) \in \bR^k: |n| \le \frac{N}{10}, | n \alp_i - a_i| \le \frac1{10}\del_i \quad (1 \le i \le k-1) \Bigr\}.
\]
Define
\[
\lam = (\del_1 \cdots \del_{k-1} N )^{1/k}, \qquad \cS = \lam^{-1} \cR.
\]
Let $\lam_1 \le \lam_2 \le \cdots \le \lam_k$ be the reduced successive minima \cite[Lecture X]{Sie1989} of the symmetric convex body $\cS$. Corresponding to these are vectors $\bv_1, \ldots, \bv_k \in \bZ^k$ whose $\bZ$-span is $\bZ^k$, and for which $\bv_i \in \lam_i \cS$ ($1 \le i \le k$). By the First Finiteness Theorem \cite[Lecture X, \S 6]{Sie1989}, we have
\begin{equation} \label{FFT}
\lam_1 \cdots \lam_k \asymp_k \vol(\cS)^{-1} \asymp 1.
\end{equation}
We choose moduli parameters $A_i = |\pi_1(\bv_i)|$ ($1 \le i \le k$). As
\[
\det(\bv_1,\ldots, \bv_k) = \pm 1,
\]
we must have \eqref{coprime}.

Next, we bound $\lam_1$ from below. We know that
\[
\bv_1 \in \lam_1 \cS = \frac{\lam_1}{\lam} \cR
\]
has integer coordinates, so with $n = |\pi_1(\bv_1)|$ we have 
\[
1 \le n \le \frac{\lam_1}{10 \lam} N, \qquad \| n \alp_i \| \le \frac{\lam_1}{10 \lam} \del_i \quad (1 \le i \le k-1),
\]
and so
\[
\| n \alp_1 \| \cdots \| n \alp_{k-1} \| \ll (\lam_1/\lam)^{k-1} \del_1 \cdots \del_{k-1} \ll (\lam_1/\lam)^{k-1}.
\]
On the other hand
\[
\| n \alp_1 \| \cdots \| n \alp_{k-1} \| \gg n^{\eps- \omega^\times(\balp)} \gg (N \lam_1/\lam)^{\eps-\omega^\times(\balp)}.
\]
Together, the previous two inequalities yield
\[
(\lam_1/\lam)^{k-1 +\omega^\times(\balp)-\eps} \gg N^{\eps - \omega^\times(\balp)},
\]
and therefore
\[
\lam_1 \gg \lam N^{\frac{\eps-\ome^\times(\balp)}{k-1 +\omega^\times(\balp) -\eps}}
\gg N^{\frac{1- \eps(k-1)}k + \frac{\eps-\ome^\times(\balp)}{k-1 +\omega^\times(\balp) -\eps}}.
\]

This enables us to bound $\lam_k$ from above: from \eqref{FFT}, we have
\[
\lam_k \ll \lam_1^{1-k} \ll N^{(k-1) \bigl( \frac{\ome^\times(\balp) - \eps}{k-1 +\omega^\times(\balp) -\eps} \: - \: \frac{1-\eps(k-1)}k \bigr)}.
\]
As $\eps$ is small and $\ome^\times(\balp) < \frac{k-1}{k-2}$, the exponent is strictly less than 
\[
(k-1) \Bigl( \frac{ \frac{k-1}{k-2} } {k-1 + \frac{k-1}{k-2}} \: - \: \frac1k \Bigr) - 2 \eps = \frac1k - 2 \eps.
\]
(We interpret the left hand side as a limit if $k=2$.) Since
\[
\lam \gg N^{\frac{1-\eps(k-1)}k},
\]
with $\eps$ small and $N$ large, we conclude that $\lam \ge k \lam_k ( N^{\eps} + 1)$. We now specify our length parameters
\[
N_i =  \Bigl \lfloor \frac{\lam}{k \lam_i} \Bigr \rfloor \ge N^\eps\qquad (1 \le i \le k).
\]

For $i=1,2,\ldots, k$, we have $\bv_i \in \frac{\lam_i}\lam \cR \cap \bZ^k$. The triangle inequality now ensures that
\begin{equation} \label{homog}
\tilde P(\bv_1, \ldots, \bv_k; N_1, \ldots, N_k) \subseteq \tilde B_0.
\end{equation}

\subsection{Finding and adjusting a base point}

By \eqref{NotDualLiouville} and Theorem \ref{BL}, together with the fact that $\eps$ is small, we have $\hat \ome(\balp, \bgam) > \eps$. Hence, there exists $b_0 \in \bN$ such that
\[
b_0 \le \frac N{20}, \qquad \| b_0 \alp_i - \gam_i \| \le \frac{\del_i}{20} \quad (1 \le i \le k-1).
\]
By Dirichlet's approximation theorem \cite[Theorem 4.1]{BRV2016}, choose $s \in \bN$ such that 
\[
s \le \lfloor N/20 \rfloor, \qquad \| s \alp_i \| \le \lfloor N/20 \rfloor^{-1/(k-1)} \quad (1 \le i \le k-1).
\]
As $\omega(\balp) \le \frac{\ome^\times(\balp)}{k-1} < \frac1{k-2}$, and since $\eps$ is small and $N$ large, we must also have 
\[
s \ge N^{\frac{k-2+ \sqrt \eps}{k-1}} \ge N^{\sqrt \eps}.
\]
We modify our basepoint by putting $b := b_0 + s$. By the triangle inequality, this ensures that
\[
N^{\sqrt \eps} \le b \le \frac N{10}, \qquad \| b \alp_i - \gam_i \| \le \frac{\del_i}{10}  \quad (1 \le i \le k-1).
\]
With the base point, moduli parameters, and length parameters specified, we have how defined the generalised arithmetic progression 
\[
P = P^+(b,A_1, \ldots, A_k, N_1, \ldots, N_k).
\]

\subsection{Projection, properness, and size}

First and foremost, we verify the inclusion $P \subseteq B$. Any $n \in P$ has the shape
\[
n = b + \sum_{i \le k} n_i \pi_1(\bv_i)
\]
for some integers $n_1 \in [-N_1, N_1]$, \ldots, $n_k \in [-N_k, N_k]$. By \eqref{homog} and the triangle inequality, we have
\[
|n| \le b + \sum_{i \le k} N_i  A_i \le \frac N {10} + \frac N {10} < N
\]
and, for $i=1,2,\ldots,k-1$,
\[
\| n \alp_i - \gam_i \| \le \| b \alp_i - \gam_i \| + \Bigl\| \pi_1 \Bigl(\sum_j  n_j \bv_j \Bigr) \alp_i \Bigr \| \le \frac{\del_i}{10} + \frac{\del_i}{10} < \del_i.
\]
We conclude that $P \subseteq B$.

Second, we show that $P$ is proper. Suppose that integers $n_i, m_i \in \{ 1,2,\ldots, N_i \}$ ($1 \le i \le k$) satisfy
\[
b + n_1 A_1 + \cdots + n_k A_k = b + m_1 A_1 + \cdots + m_k A_k.
\]
Then, with $(x_1, \ldots, x_k) = (n_1,\ldots,n_k) - (m_1, \ldots, m_k)$, we have 
\[
\sum_{i \le k} x_i | \pi_1(\bv_i)| = 0.
\]
With $y_i = x_i \cdot \sgn( \pi_1(\bv_i) )$ and $\by = (y_1, \ldots, y_k)^T$, we now have
\[
\pi_1(M \by) = 0,
\]
where $M = (\bv_1, \ldots, \bv_k) \in \mathrm{GL}_k(\bZ)$. Moreover
\[
M\by \in \tilde P(\bv_1, \ldots, \bv_k; N_1, \ldots, N_k) \subseteq \tilde B_0,
\]
so we draw the \emph{a priori} stronger conclusion that $M \by = \bzero$. As $M$ is invertible, we obtain $\by = \bzero$, so $\bx = \bzero$, and we conclude that $P$ is proper.

Finally, as $P$ is proper, its cardinality is readily computed as
\[
|P| = N_1 \cdots N_k \gg \prod_{i \le k} (\lam/ \lam_i)^k \gg \lam^k = \del_1 \cdots \del_{k-1}N .
\]
This completes the proof of Lemma \ref{inner}.

\subsection{Structure outside Bohr sets, and an upper bound on the cardinality}

In this subsection we provide an `outer' construction, complementing the structural lemma of the previous subsection. For the purpose of Theorem \ref{main}, we only require this for homogeneous Bohr sets (those with $\bgam = \bzero$). A standard counting trick will then enable us to handle the shift $\bgam$, accurately bounding the size of $B_\balp^\bgam(N;\bdel)$. Put $\tau = \sqrt \eps$.

\begin{lemma} [Outer structure] \label{outer} If
\[
N^{-\tau} \le \del_i \le 2 \qquad (1 \le i \le k-1)
\]
then there exists a generalised arithmetic progression
\[
P' = P(0; A_1, \ldots, A_k; N_1, \ldots, N_k)
\]
containing $B_\balp^\bzero(N;\bdel)$, for which $|P'| \ll \del_1 \cdots \del_{k-1}N$.
\end{lemma}

\begin{proof} We initially follow the proof of Lemma \ref{inner}, with $\tau$ in place of $\eps$. Now, however, we enlarge the $N_i$ by a constant factor: let $C_k$ be a large positive constant, and choose $N_i = \lfloor C_k \lam / \lam_i \rfloor \ge N^\tau$ for $i=1,2,\ldots,k$. The cardinality of $P'$ is bounded above as
\[
|P'| \ll_k N_1 \cdots N_k \ll \del_1 \cdots \del_{k-1} N,
\]
so our only remaining task is to show that $B_\balp^\bzero(N;\bdel) \subseteq P'$. We establish, \emph{a fortiori}, that $\tilde B_\balp^\bzero(N;\bdel) \subseteq \tilde P'$. 

Let $(n, a_1,\ldots,a_{k-1}) \in \tilde B_\balp^\bzero(N;\bdel)$. Since $\bv_1,\ldots, \bv_k$ generate $\bZ^k$, there exist $n_1,\ldots,n_k \in \bZ$ such that
\[
\bn := (n, a_1, \ldots, a_{k-1})^T = n_1 \bv_1 + \cdots + n_k \bv_k.
\]
Let $M = (\bv_1,\ldots,\bv_k) \in \mathrm{GL}_k(\bZ)$, and for $i=1,2,\ldots,k$ let $M_i$ be the matrix obtained by replacing the $i$th column of $M$ by $\bn$. Now Cramer's rule gives
\[
|n_i| = |\det(M_i)|.
\]
Observe that $\bn \in 10 \lam \cS$ and $\bv_i \in \lam_i \cS$. Determinants measure volume, so by \eqref{FFT} we have
\[
n_i \ll |\lam \lam_1 \cdots \lam_k / \lam_i| \ll_k | \lam / \lam_i|.
\]
As $C_k$ is large, we have $|n_i| \le N_i$, and so $\bn \in \tilde P'$.
\end{proof}

\begin{cor} [Cardinality bound] \label{card} If 
\[
N^{-\tau} \le \del_i \le 1 \qquad (1 \le i \le k-1)
\]
then
\[
\# B_\balp^\bgam(N;\bdel) \ll \del_1 \cdots \del_{k-1} N.
\]
\end{cor}

\begin{proof} We may freely assume that $B_\balp^\bgam(N;\bdel)$ is non-empty. Fix $n_0 \in B_\balp^\bgam(N;\bdel)$. By the triangle inequality, the function $n \mapsto n-n_0$ defines an injection of $B_\balp^\bgam(N;\bdel)$ into $B_\balp^\bzero(N; 2\bdel)$, so
\[
|B_\balp^\bgam(N;\bdel)| \le \max\{1, |B_\balp^\bzero(N; 2\bdel)| \}.
\]
An application of Lemma \ref{outer} completes the proof.
\end{proof}

\section{The preponderance of reduced fractions}
\label{average}

In this section, we use the generalised arithmetic progression structure to control the average behaviour of the Euler totient function $\varphi$ on 
\[
\hat B_\balp^\bgam (N; \bdel) := B_\balp^\bgam (N; \bdel) \cap [N^{\sqrt \eps}, N].
\]
The AM--GM inequality \cite[Ch. 2]{Ste2004} will enable us to treat each prime separately, at which point we can employ the geometry of numbers.

\begin{lemma} [Good averaging]\label{averagegood}
Let $N^{- \eps} \le \del_1, \ldots, \del_{k-1} \le 1$. Then
\[
\sum_{n \in \hat B_\balp^\bgam (N; \bdel) } \frac{\varphi(n)}n \gg \del_1 \cdots \del_{k-1} N.
\]
\end{lemma}

\begin{proof}
Let $P \subset \hat B_\balp^\bgam (N; \bdel)$ 
denote the generalised arithmetic progression from Lemma \ref{inner}.
Since
\[
\sum_{n \in \hat B_\balp^\bgam (N; \bdel) } \frac{\varphi(n)}n \geq
\sum_{n \in P } \frac{\varphi(n)}n,
\]
and since $|P| \gg \del_1 \cdots \del_{k-1}N$, 
the AM--GM inequality implies that it suffices to establish that 
\begin{equation}\label{defofX}
X := \Biggl( \prod_{n\in P } \frac{\varphi (n)}{n} 
\Biggr)^{\frac{1}{\vert P \vert }} \gg 1.
\end{equation}
To this end, we observe that the well-known relation
\[
\frac{\varphi (n)}{n} = \prod_{p \mid n} (1- 1/p)
\]
permits us to rewrite $X$ as 
\[
X = \prod_{p \le N} (1- 1/p)^{\alp_{p}},
\]
where $\alpha_{p} = \vert P \vert ^{-1} \vert 
\{n\in P: n \equiv 0 \mmod p \} \vert $.
It therefore remains to show that
\begin{equation}\label{alphasmall}
\alp_{p} \ll p^{- \varepsilon}.
\end{equation}
Indeed, once we have (\ref{alphasmall}) at hand, we can infer that
\[
\log(1/X) \le \sum_p \alp_p \log(1+2/p) \ll \sum_p p^{-\eps} \log(1+2/p),
\]
whereupon the trivial inequality $\log(1+2/p) \le 2/p$ yields
\[
\log(1/X) \ll \sum_p p^{-1-\eps} \ll 1,
\]
implying (\ref{defofX}).

We proceed to establish \eqref{alphasmall}. We may suppose that $\alp_p > 0$, which allows us to fix positive integers $n_1^{*} \le N_1, \ldots, n_k^* \le N_k$ for which
\[
b+ A_1 n_1^* + \cdots + A_{k} n_{k}^* \equiv 0 \mmod p.
\]
Then
\begin{equation} \label{cong}
A_1 n'_1 + \cdots + A_{k} n'_{k} \equiv 0 \mmod p,
\end{equation}
where $n'_i = n_i - n_i^*$ ($1 \le i \le k$) are integers such that $(n'_1, \ldots, n'_k)$ lies in the box
\[
\cB := [-N_1, N_1] \times \cdots \times [-N_{k}, N_{k}] \subseteq \bR^{k}.
\]
In particular, the quantity $|P| \alp_p$ is bounded above by the number of integer 
solutions to $\eqref{cong}$ in the box $\cB$. 

Let $\mathcal{J}$ denote the set of $i\in \{1, \ldots, k \}$ 
such that $p \mid A_{i}$, and let $\mathcal{J}^{c}$ 
be its complement in $\{1, \ldots, k \}$.
We note from \eqref{coprime} that
\[
\vert \mathcal{J} \vert \leq k-1.
\]
Thus, the number $\cN$ of solutions to (\ref{cong}) is at most 
$\displaystyle \prod_{i \in \mathcal{J}} (2 N_{i} + 1)$ times 
the number of integer vectors
$(n''_{1},\ldots,n''_{\vert \mathcal{J}^{c} \vert})$ 
in the box $\cB_{\mathcal{J}} 
:= \displaystyle \prod_{i\in \mathcal{J}^c} \left[ - N_{i}, N_{i} \right] $
which, additionally, satisfy the congruence
\begin{equation} \label{reducong}
\sum_{i\in \mathcal{J}^{c}}
A_{i} n''_{i} \equiv 0 \mmod p.
\end{equation}
As (\ref{reducong}) defines a full-rank lattice in $\bR^{\vert \mathcal{J}^{c} \vert}$ 
of determinant $p$, and we can exploit a counting result due to Davenport \cite{D1951}; 
see also \cite{BW2013} and \cite[p. 244]{T1993}. Our precise statement follows from \cite[Lemmas 2.1 and 2.2]{BW2013}.

\begin{thm} [Davenport]\label{DavenportBound}
Let $d$ be a positive integer, and $\cS \subset \bR^{d}$ compact. 
Suppose the two following conditions are met:
\vspace{7pt}

(i) Any line intersects $\cS$ in a set of points which, if non-empty,
consists of at most $h$ intervals, and

\vspace{7pt}
(ii) the condition (i) holds true --- with $j$ in place of $d$ --- for 
any projection of $\cS$ onto a $j$-dimensional subspace. 

\vspace{7pt}
\noindent 
Moreover, let $\mu_1 \le \cdots \le \mu_d$ denote the successive minima, 
with respect to the Euclidean unit ball, 
of a (full-rank) lattice $\Lambda\subset \bR^{d}$. Then
\[
\biggl\vert \vert \cS \cap \Lambda \vert 
- \frac{\mathrm{vol (\cS)}}{\det \: \Lambda}
\biggr\vert 
\ll_{d,h} \sum_{j=0}^{d-1} \frac{V_{j}(\cS)}{\mu_1 \cdots \mu_{j}},
\]
where $V_{j}(\cS)$ is the supremum of the $j$-dimensional volumes of the projections 
of $\cS$ onto any $j$-dimensional subspace, 
and for $j=0$ the convention $V_{0}(\cS)=1$ is to be used.
\end{thm}

As $\mathcal{B}_{\mathcal{J}}$ satisfies the hypotheses of 
Theorem \ref{DavenportBound}, with $h=1$ and $d = |\cJ^c|$, and with each $V_{j}(\mathcal{B}_{\mathcal{J}})$ 
less than the surface area of $\mathcal{B}_{\mathcal{J}}$,
we obtain
\[
\cN \ll 
\Biggl( \prod_{i \in \mathcal{J}} N_{i} \Biggr)
	\biggr( 
		\frac{\prod_{i \in \mathcal{J}^{c}}N_{i}}{p} 
		+  \sum_{i\in \mathcal{J}^{c}} 
			\frac{\prod_{j \in \mathcal{J}^{c}}N_{j}}{N_{i}} 
	\biggl).
\]
Here we have used the fact that $\mu_d \ge \cdots \ge \mu_1 \ge 1$, which follows from our lattice being a sublattice of $\bZ^d$. Therefore
\[
\alp_{p} 
\ll
\Biggl( \prod_{i \in \mathcal{J}^{c}} N_{i} \Biggr)^{-1}
	\biggr( 
		\frac{\prod_{i \in \mathcal{J}^{c}}N_{i}}{p} 
		+  \sum_{i\in \mathcal{J}^{c}} 
			\frac{\prod_{j \in \mathcal{J}^{c}}N_{j}}{N_{i}} 
	\biggl)
\ll
\frac{1}{p} 
+ \frac{1}{\min \{ N_{i} :\, i \in \mathcal{J}^{c}\}},
\]
and Lemma \ref{inner} guarantees that
$N_{i} \geq N^{\eps} \geq p^{\varepsilon}$ 
for $i=1,\ldots,k$. Now (\ref{alphasmall}) follows, and the proof is complete.
\end{proof}

\section{Generalised sums of reciprocals of fractional parts}
\label{fractional}

As in \cite{BHV2016, Cho2018}, an essential part 
of the analysis is to estimate generalisations 
of sums of reciprocals of fractional parts.
Recall that we fixed real numbers $\alp_1,\ldots,\alp_{k-1}$ 
and $\gam_1,\ldots, \gam_{k-1}$, 
with $\ome^\times(\alp_1,\ldots,\alp_{k-1})< \frac{k-1}{k-2}$, 
from the beginning. As in \cite{Cho2018}, 
we restrict the range of summation. Let
\[
G = \{ n \in \bN: \| n \alp_i - \gam_i \| 
\ge n^{- \sqrt \eps} \quad (1 \le i \le k-1) \}.
\]
We consider the sums $T_N(\balp, \bgam)$ and $T_N^*(\balp,\bgam)$ 
defined in \eqref{T}, \eqref{Tstar}, and show that
\begin{equation} \label{order}
T_N(\balp, \bgam) \asymp  T_N^*(\balp, \bgam)  \asymp N (\log N)^{k-1}.
\end{equation}

We begin with an upper bound.

\begin{lemma} \label{upper} We have
\[
T_N(\balp, \bgam) \ll N (\log N)^{k-1}.
\]
\end{lemma}

\begin{proof}
First, we decompose $T_N(\balp, \bgam)$ so that 
the size parameters $\delta_{i}$ determine dyadic ranges: 
\[
T_N(\balp, \bgam) =  \sum_{i_{1},\ldots,i_{k-1} \in \bZ} \quad
\sum_{\substack{n \le N, n\in G \\ 
\, 2^{-(i_{j}+1)}      
<   \| n \alp_j - \gam_j \|   \le 2^{-i_{j}}}} 
\frac1{\| n \alp_{1} - \gam_{1}\| \cdots \| n \alp_{k-1} - \gam_{k-1}\|}.
\]
For each $j\leq k-1$ there are $O(\log N)$ choices of $i_{j}$ for which the inner sum is non-zero, owing to our choice of $G$. Therefore the inner sum is non-zero $O( (\log N)^{k-1})$ times. Furthermore, the inner sum is bounded above by
\[
\Bigl( \prod_{j \le k-1} 2^{i_j + 1} \Bigr) \# B_\balp^\bgam(N; 2^{-i_1}, \ldots, 2^{-i_{k-1}})
\]
which, by Corollary \ref{card}, is $O(N)$. This completes the proof.
\end{proof}

We also require a lower bound for $T_N^*(\balp, \bgam)$. 

\begin{lemma} \label{lower} We have
\[
T_N^*(\balp, \bgam) \gg N (\log N)^{k-1}.
\]
\end{lemma}

\begin{proof}
First observe that if $N^{\sqrt \eps} \le n \le N$ and 
\[
\| n \alp_i - \gam_i \| \ge N^{- \eps} \qquad (1 \le i \le k-1)
\]
then $n \in G$. It therefore suffices to prove that
\begin{equation} \label{suffices}
\sum_{\substack{N^{\sqrt \eps} \le n \le N \\ 
\| n \alp_i - \gam_i \| \ge N^{- \eps} }} \frac{\varphi(n)}{n 
\| n \alp_1 - \gam_1 \| \cdots \| n \alp_{k-1} - \gam_{k-1} \| } 
\gg N (\log N)^{k-1}.
\end{equation}
Before proceeding in earnest, we note from Lemma \ref{averagegood} and Corollary \ref{card}
that if $N^{-\eps} \leq \delta_1, \ldots, \delta_{k-1}\leq 1$ then
\begin{equation}\label{RightAverage}
\sum_{n \in \hat B_\balp^\bgam(N; \bdel)} \frac{\varphi(n)}n 
\asymp \delta_{1}\cdots\delta_{k-1} N,
\end{equation}
wherein the implied constants depend at most on $\alp_{1},\ldots, \alpha_{k-1}$.

Let $\eta$ be a constant which is small 
in terms of the constants implicit in (\ref{RightAverage}),
and put $\bdel=(\delta_1, \ldots, \delta_{k-1})$. We split the left hand side of (\ref{suffices}) 
into $\gg_\eta (\log N)^{k-1}$ sums, for which $N^{-\eps} \le \del_1, \ldots, \del_{k-1} \le 1$, of the shape
\begin{equation} \label{SplitAvarage}
\sum_{n \in \hat B_\balp^\bgam (N; \bdel)  
\setminus \hat B_\balp^\bgam (N; \eta \bdel) } 
\frac{\varphi(n)}
 {n \| n \alp_1 - \gam_1 \| \cdots \| n \alp_{k-1} - \gam_{k-1} \|}.
\end{equation}
Each of these sums exceeds $\eta^{k-1} (\delta_1 \cdots \delta_{k-1})^{-1}$ times
\[
\sum_{n \in \hat B_\balp^\bgam (N; \bdel)  
\setminus \hat B_\balp^\bgam (N; \eta \bdel) } 
\frac{\varphi(n)}{n}
= 
\sum_{n \in \hat B_\balp^\bgam (N; \bdel)} 
\frac{\varphi(n)}{n} 
- \sum_{n \in \hat B_\balp^\bgam (N; \eta \bdel)} 
\frac{\varphi(n)}{n}.
\]
Since the right hand side is
$ \gg \delta_1 \cdots \delta_{k-1} N$, by \eqref{RightAverage} and $\eta$ being small,
we conclude that the quantity (\ref{SplitAvarage}) is $\gg_\eta N$. 
As $\eta$ need only depend on $\balp$, this entails (\ref{suffices}), and thus completes the proof.
\end{proof}

As $T_N(\balp, \bgam) \ge T^*_N( \balp, \bgam)$, the previous two lemmas imply \eqref{order}.

\section{An application of the Duffin--Schaeffer theorem}
\label{final}

In this section, we finish the proof of Theorem \ref{main}. The overall strategy is to apply the Duffin--Schaeffer theorem (Theorem \ref{DSthm}) to the approximating function
\[
\Psi(n) = \Psi_\balp^\bgam(n) = \begin{cases} \frac{\psi(n)}{\| n \alp_1 - \gam_1 \| 
\cdots \| n \alp_{k-1} - \gam_{k-1} \|}, &\text{if }  n \in G \\
0, &\text{if } n \notin G.
\end{cases}
\]
A valid application of the Duffin--Schaeffer theorem will complete the proof, so we need only verify its hypotheses, namely
\begin{equation} \label{hyp1}
\sum_{n=1}^\infty \frac {\varphi(n)}n \Psi(n) = \infty
\end{equation}
and
\begin{equation} \label{hyp2}
\sum_{n \le N} \frac {\varphi(n)}n \Psi(n) \gg \sum_{n \le N} \Psi(n).
\end{equation}
The inequality \eqref{hyp2} is only needed 
for an infinite strictly increasing sequence of positive integers $N$, 
but we shall prove \emph{a fortiori} that for all large $N$ we have
\begin{equation} \label{hyp21}
\sum_{n \le N} \frac {\varphi(n)}n \Psi(n)  
\gg  \sum_{n \le N} \psi(n) (\log n)^{k-1}
\end{equation}
and
\begin{equation} \label{hyp22}
\sum_{n \le N} \Psi(n) \ll 
\sum_{n \le N} \psi(n) (\log n)^{k-1}.
\end{equation}
Observe, moreover, that \eqref{divergence} 
and \eqref{hyp21} would imply \eqref{hyp1}. 
The upshot is that it remains 
to prove \eqref{hyp21} and \eqref{hyp22}.

Recall the sums $T_N(\balp, \bgam)$ and $T_N^*(\balp, \bgam)$ considered in the previous section, and let $N_0 \in \bN$ 
be a large constant. By partial summation and the fact that 
\[
\psi(n) \ge \psi(n+1),
\]
we have the lower bound
\[
\sum_{n \le N}  \frac {\varphi(n)}n \Psi(n) \ge
\psi(N+1) T_N^*(\balp, \bgam) + \sum_{n = N_0}^N 
(\psi(n) - \psi(n+1)) T_n^*(\balp,\bgam).
\]
Applying Lemma \ref{lower} to continue our calculation yields
\[
\sum_{n \le N}  \frac {\varphi(n)}n \Psi(n) \gg 
\psi(N+1) N (\log N)^{k-1} + \sum_{n = N_0}^N 
(\psi(n) - \psi(n+1)) n (\log n)^{k-1}.
\]
As $\psi(n) \ge \psi(n+1)$ and 
$\displaystyle \sum_{m \le n} (\log m)^{k-1} \le n (\log n)^{k-1}$, 
we now have
\[
\sum_{n \le N}  \frac {\varphi(n)}n \Psi(n) \gg \psi(N+1) 
\sum_{m \le N} (\log m )^{k-1} + \sum_{n = N_0}^N (\psi(n) - \psi(n+1)) 
 \sum_{m \le n} ( \log m )^{k-1}.
\]
Another application of partial summation now gives
\[
\sum_{n \le N}  \frac {\varphi(n)}n \Psi(n) 
\gg \sum_{n = N_0}^N \psi(n) (\log n)^{k-1},
\]
establishing \eqref{hyp21}.

We arrive at the final piece of the puzzle, which is \eqref{hyp22}. By partial summation, we have
\[
\sum_{n \le N}  \Psi(n)  = \psi(N+1)T_N(\balp, \bgam)
+ \sum_{n \le N} (\psi(n) - \psi(n+1)) T_n(\balp,\bgam).
\]
Observe that if $n \le N_0$ 
then $T_n(\balp, \bgam) \le T_{N_0}(\balp, \bgam) \ll 1$. 
Thus, applying Lemma \ref{upper} to continue our calculation yields
\[
\sum_{n \le N}  \Psi(n) 
\ll 1 + \psi(N+1)N (\log N)^{k-1} 
+ \sum_{n = N_0}^N (\psi(n) - \psi(n+1)) n (\log n)^{k-1}.
\]
Partial summation tells us that 
$ \displaystyle \sum_{m \le n} (\log m)^{k-1} 
\gg n (\log n)^{k-1}$, 
and so
\[
\sum_{n \le N}  \Psi(n) 
\ll 1 + \psi(N+1) \Bigl(\sum_{m \le N} (\log m)^{k-1} \Bigr) 
+ \sum_{n = N_0}^N (\psi(n) - \psi(n+1)) \sum_{m \le n} (\log m)^{k-1}.
\]
A further application of partial summation now gives
\[
\sum_{n \le N}  \Psi(n) \ll 1 + \sum_{n = N_0}^N \psi(n) (\log n)^{k-1}.
\]
This confirms \eqref{hyp22}, thereby completing the proof of Theorem \ref{main}.

\providecommand{\bysame}{\leavevmode\hbox to3em{\hrulefill}\thinspace}

\end{document}